\documentclass[11pt,reqno]{amsart}
\usepackage{amssymb,xfrac}
\usepackage{amsmath}
\usepackage{color}
\usepackage{graphics, cite, bookmark}
\usepackage{epsfig}
\usepackage{hyperref}
\usepackage{epstopdf}
\newtheorem{theorem}{Theorem}[section]
\newtheorem*{theorem*}{Theorem}
\newtheorem{lemma}{Lemma}[section]

\newtheorem{proposition}{Proposition}[section]

\theoremstyle{definition}

\theoremstyle{remark}
\newtheorem{remark}{Remark}[section]

\numberwithin{equation}{section}
\allowdisplaybreaks

\newcommand{\R}{\ensuremath{\mathbb{R}}}

\newcommand{\N}{\ensuremath{\mathbb{N}}}

\newcommand{\id}{\ensuremath{\mathrm{Id}}}
\newcommand{\sd}{\ensuremath{\mathrm{sym}\nabla}}
\newcommand{\sym}{\ensuremath{\mathrm{sym}}}
\newcommand{\curl}{\mathrm{curl}}

\renewcommand{\S}{\ensuremath{\mathcal{S}}}

\usepackage{bbm}

\begin{document}

\title[one third]
{$C^{1,\sfrac{1}{3}-}$ very weak solutions to the two dimensional Monge--Amp\`ere equation}

\author[W. Cao, J. Hirsch \& D. Inauen]{Wentao Cao, Jonas Hirsch, and Dominik Inauen}

\address{Wentao Cao, Academy for Multidisciplinary Studies, Capital Normal University, West 3rd Ring North Road 105, Beijing, 100048 P.R. China. E-mail:{\tt cwtmath@cnu.edu.cn}}
\address{Jonas Hirsch, Institut f\"{u}r Mathematik, Universit\"{a}t Leipzig, D-04109, Leipzig, Germany.  E-mail:{\tt jonas.hirsch@math.uni-leipzig.de}}
\address{Dominik Inauen, Institut f\"{u}r Mathematik, Universit\"{a}t Leipzig, D-04109, Leipzig, Germany.  E-mail:{\tt dominik.inauen@math.uni-leipzig.de}}

\begin{abstract}
For any $\theta<\frac{1}{3}$, we show that very weak solutions to the two-dimensional Monge--Amp\`ere equation with regularity $C^{1,\theta}$ are dense in the space of continuous functions. This result is shown by a convex integration scheme involving a subtle decomposition of the defect at each stage. The decomposition diagonalizes the defect and, in addition, incorporates some of the leading-order error terms of the first perturbation, effectively reducing the required amount of perturbations to one. 
\end{abstract}

\date{\today}



\maketitle 
\section{Introduction} This work is concerned with very weak solutions of the Monge--Amp\`ere equation on a simply connected, bounded open set $\Omega\subset\R^{2}$
\begin{equation}\label{e:MA}
\mathcal D\text{et} D^{2}v = f  \text{ in } \Omega \,.
\end{equation}
Here, $\mathcal D\text{et} D^{2}v $ is the distribution 
\[\mathcal D\text{et} D^{2}v := \partial^{2}_{12}\left (\partial_{1} v\partial_{2}v\right ) -\frac{1}{2}\partial_{22}^{2}\left ((\partial_1v)^{2}\right ) -\frac{1}{2}\partial_{11}^{2}\left ((\partial_2v)^{2}\right )=-\frac{1}{2}\mathrm{curl \,curl}(\nabla v\otimes \nabla v)\,,\]
introduced in \cite{Iwaniec} as the \emph{very weak Hessian}. It is well defined for $v\in W^{1,2}_{loc}(\Omega)$ and a quick approximation argument shows that it agrees with $\det D^2v$ for $v \in W^{2,2}_{loc}(\Omega)$. As in \cite{CaoSz,LewickaPakzad}, distributional solutions to \eqref{e:MA} are called \emph{very weak solutions} of the Monge-Amp\`ere equation.
The following is our main result. 
\begin{theorem}\label{t:main} Let $\Omega\subset \R^{2}$ be an open, bounded, simply connected and smooth domain and $f\in L^{p}(\Omega)$ for $p\ge \frac{3}{2}$. For any $\theta<\frac{1}{3}$ the set of $C^{1,\theta}(\bar\Omega)$ solutions of \eqref{e:MA} is dense in $C^{0}(\bar \Omega)$. More precisely, for any $\underline {v}\in C^{0}(\bar \Omega)$ and $\epsilon >0$ there exists  distributional solution $\overline{v}\in C^{1,\theta}(\bar \Omega)$ of \eqref{e:MA} satisfying 
\[ \|\underline {v} - \overline {v}\|_{0}<\epsilon\,.\]
\end{theorem}

This result improves upon the main theorems in \cite{LewickaPakzad} and \cite{CaoSz}, where the above theorem is shown for $\theta<\frac{1}{7}$ and $\theta<\frac{1}{5}$ respectively. It shows a type of flexibility of the space of solutions to \eqref{e:MA} exhibited more generally by differential relations satisfying the $h$-principle, a term introduced in \cite{Gromov} and proven to be present in various problems across fields such as geometry, topology, PDE and also fluid dynamics (see for example \cite{Gromov}, \cite{Eliashberg},\cite{Laszlo},\cite{BuckmasterVicol}). 
In most of these instances (and in particular for the main theorems of \cite{LewickaPakzad,CaoSz} and the present work), this flexibility is shown by means of a convex integration scheme, a method originally introduced by J. Nash in \cite{Nash54} for the construction of counter-intuitive isometric embeddings of Riemannian manifolds into Euclidean space.

\
\subsection{Connection to isometric embeddings and reduction to convex integration} Recall that an isometric embedding of a Riemannian manifold $(\mathcal M,g)$ into $\R^{m}$ is an injective $C^{1}$-map $u:\mathcal M\to \R^{m}$ such that the induced metric $u^{\sharp}e$ agrees with $g$. In local coordinates this leads to the system 
\[ \partial_i u \cdot \partial_j u = g_{ij}\,,1\leq i,j\leq \mathrm{dim}\mathcal M\]
sometimes simply written $\nabla u^{T}\nabla u = g$. The connection between isometric embeddings of two-dimesional Riemannian manifolds into $\R^{3}$ and the Monge--Amp\`ere equation is intimate: finding a (local) $C^{2}$-regular isometric embedding of a two-dimensional domain into $\R^{3}$ is equivalent to solving the Darboux equation, a fully nonlinear equation of Monge--Amp\`ere type (see for example \cite{HanHong}). 
As observed in \cite{LewickaPakzad}, a connection remains on the level of weaker solutions, i.e., isometric embeddings of regularity $C^{1,\theta}$ and very weak solutions of the Monge-Amp\`ere equation of regularity $C^{1,\theta}$ respectively. 
Indeed, recall that the kernel of the operator curlcurl on a simply connected domain is given by matrices of the form $\sd w :=\frac{1}{2}(\nabla w + \nabla w^{T})$ for $w:\Omega\to\R^{2}$. Thus, solving \eqref{e:MA} is equivalent to solving the system 
\begin{align} 
 \frac{1}{2}\nabla v\otimes \nabla v +\sd w = A \label{e:cI}\\
 -\mathrm{curl\,curl} A = f  \label{e:f}
\end{align}
in the unknowns $v:\Omega\to \R$, $w:\Omega\to \R^{2}$ and $A:\Omega\to \R^{2\times 2}_{sym}$.
The second equation is simply a compatibility condition and is easily solved: one can set $A = u\, \id $ for a function $u$ such that $-\Delta u = f$. 
On the other hand, for $A=0$, equation \eqref{e:cI} can be seen as the condition for the family $\phi_t:\Omega\to \R^{3}$ given by $\phi_t(x) = (x+t^{2}w(x), tv(x))$ to "preserve lengths upto second order" (i.e., the induced metric $\nabla \phi_t^{T}\nabla \phi_t$ agrees with the identity metric upto terms of order $t^{2}$), or in other words, to form a second order infinitesimal isometry\footnote{For a more elaborate explanation of this connection and further motivations to study very weak solutions to the Monge-Amp\`ere equation, see for example the introductions in \cite{LewickaPakzad} and \cite{CaoSz}.}. Motivated by this connection, Lewicka--Pakzad observed in \cite{LewickaPakzad} that an adaptation of the convex integration process of \cite{CDS} for the construction of $C^{1,\theta}$ isometric embeddings can be used to construct solutions of \eqref{e:cI}. Such a process generates a solution as the limit of a sequence of "subsolutions", i.e., in the present case, maps $(v_{q},w_q)$ such that 
\[ \mathcal D_q:= A-\frac{1}{2}\nabla v_q\otimes \nabla v_q-\sd w_q\]
is positive definite. The construction of $(v_{q+1}, w_{q+1})$ given $(v_q,w_q)$, termed "stage", aims at decreasing the defect $\mathcal D_q$ and consists in first \emph{decomposing} the defect into a sum of rank-one matrices (called "primitive metrics" in \cite{Nash54} and subsequent works) and then \emph{perturbing} $(v_q,w_q)$ by adding "corrugations" which cancel these primitive metrics upto an error of lower order. The corrugations oscillate at a high frequency which increases with $q$ but have amplitudes which decrease with $q$. If set-up correctly, the sequence $\|v_{q+1}-v_q\|_1$ tends to zero while $\|v_{q+1}-v_q\|_2$ blows up in a controlled way. This yields the convergence in some H\"older space $C^{1,\theta}$ to a solution of the problem. The number of corrugations used in a stage influences the size of $\theta$ (see e.g. \cite{Laszlo} for detailed explanation): the three corrugations used in \cite{LewickaPakzad}, \cite{CDS} yield a regularity $\theta<\frac{1}{7}$. For the isometric embedding problem, a change of coordinates was introduced in \cite{DIS} to diagonalize the defect. This reduced the required number of corrugations to two, allowing the construction of local isometric embeddings with regularity $C^{1,\theta}$ for $\theta<\frac{1}{5}$ (a result which was extended to global isometric embeddings in \cite{CaoSze2022}). For the Monge--Amp\`ere equation it was observed in \cite{CaoSz} that instead of introducing a change of coordinate one can use the term $\sd w$ to diagonalize the defect  directly by solving 
\begin{equation} \label{e:de-0} \mathcal D-\sd w = r \id \end{equation} 
for $w$ and $r$, where $\mathcal D$ represents the defect and $r$ is the desired coefficient function.  This gives the reduction to two corrugations and therefore regularity $\theta<\frac{1}{5}$. 

\subsection{Strategy of proof} As in \cite{LewickaPakzad}, \cite{CaoSz} we prove Theorem \ref{t:main} by constructing solutions to \eqref{e:cI} by the convex integration scheme outlined above. However, to get to H\"older regularity $\theta<\frac{1}{3}$ we need to decrease the rate with which the frequencies of the oscillations grow along the sequence. We achieve this by a decomposition which simultaneously  diagonalizes the defect and in addition incorporates (some of) the leading-order error terms of the first perturbation (see \eqref{e:de-1}). Even though we still need to add two perturbations to cancel the defect, their frequency can be chosen smaller, effectively reducing the number of perturbations to one. 

The idea of absorbing the error terms coming from the oscillations into the decomposition first appears in \cite{Kaellen} in the context of isometric embeddings into Euclidean space with high codimension, where the author makes use of the freedom of high codimension to absorb such errors  with arbitrary precision. This allows for the construction of isometric embeddings of class $C^{1,\theta}$ for all $\theta<1$, given that the metric $g$ is regular enough and the codimension high enough (this idea is also used for less regular metrics and smaller codimension in \cite{CaoSz2023, DI2020, CaoIn2020}).

Since in the present context we do not have the freedom of  codimension, the best we can hope for is to absorb the error terms of the first oscillation. In practice this involves solving an equation of the form
 \begin{equation}\label{e:de-1}
\mathcal D -\sd w +E(r) = r \id
\end{equation}
for $w$ and $r$, where $E(r)$ represents the error terms of the first perturbation (compare to \eqref{e:de-0}). 
However, due to the structure of these error terms, a direct analogue of the decomposition in \cite{Kaellen} does not yield the required improvement\footnote{The same is true for codimension one isometric embeddings.}. Instead, the observation that the fast oscillation of the first perturbation only occurs in one direction leads us to absorb only specific parts of the error terms in the decomposition. Their algebraic structure is chosen such that it is possible to trade one ``fast'' derivative for a ``slow'' one.
Equation \eqref{e:de-1} is then approached by a Picard scheme.  This yields the crucial decomposition Lemma \ref{l:decomposition}, which is one of the main technical contributions of this paper. The leading-order error terms of the first perturbation which are not absorbed by the decomposition are then matched exactly by the second perturbation. 

\subsection{Organization of the paper} After some preliminaries on H\"older spaces and mollification in Section \ref{s:preliminaries} we show in Section \ref{s:stage} how the decomposition Lemma \ref{l:decomposition} leads to a proof of the crucial inductive building block, the "stage" proposition. We then show in Section \ref{s:proofoftmain} how the stage proposition can be applied to prove the main theorem. Finally we prove the decomposition lemma in Section \ref{s:decom}.

\section{Preliminaries} \label{s:preliminaries}
In this section we introduce some notations, function spaces and basic lemmas which are needed for the proof of Theorem \ref{t:main}

\subsection{H\"older norms and interpolation} Let $\Omega \subset \R^{2}$ be as above. In the following, the maps $f$ can be real valued, vector valued, or tensor valued. In every case, the target is equipped with the Euclidean norm, denoted by $|f(x)|$. The H\"older norms are then defined as follows:
\begin{equation*}
\|f\|_0=\sup_{\Omega}|f|, ~~\|f\|_m=\sum_{j=0}^m\max_{|I|=j}\|\partial^I f\|_0,,
\end{equation*}
where $I$ denotes a multi-index,
and, for $\alpha\in ]0,1]$ and $k\in \N$, 
\begin{equation*}
[f]_{0,\alpha}=\sup_{x\neq y}\frac{|f(x)-f(y)|}{|x-y|^\alpha},~~[f]_{k,\alpha}=\max_{|I|=k}\sup_{x\neq y}\frac{ |\partial^I f(x)-\partial^I f(y)|}{|x-y|^\alpha}\,.
\end{equation*}
Then the H\"older norms are given as
$$\|f\|_{k,\alpha}=\|f\|_k+[f]_{k,\alpha}. $$
We sometimes abbreviate $[f]_{k,\alpha},\|f\|_{k,\alpha}$ as $[f]_{k+\alpha}$ and $\|f\|_{k+\alpha}$ respectively.
We recall the standard interpolation inequality
\begin{equation}\label{e:interpolation}
[f]_p\leq C\|f\|_0^{1-\frac{p}{d}}[f]_d^{\frac{p}{d}}
\end{equation}
for any $d>p\geq0$ and the Leibniz rule
\begin{equation}\label{e:Leibniz}
\|fg\|_p\leq C(\|f\|_p\|g\|_0+\|f\|_0\|g\|_p),
\end{equation}
where the constants only depend on $p,d$ and $p$ respectively. 
We also collect two classical estimates on the H\"older norms of compositions. A proof can be found in \cite{DIS}.

\begin{lemma}\label{l:chain}
Let $k,n\in \N$, $\alpha\in [0,1]$, $U\subset \R^{n}$ and $\Omega \subset \R^{2}$ as above. Assume $\Psi: U \to \mathbb R$ and $u: \Omega  \to U$ are two $C^{k,\alpha}$ functions.
Then there is a constant $C$ (depending only on $k$, $\alpha$, 
$\Omega$ and $U$) such that
\begin{align}
\left[\Psi\circ u\right]_{k,\alpha}&\leq C[u]_{k,\alpha}\left( [\Psi]_1+\|u\|_0^{k-1}[\Psi]_k\right) \nonumber \\
&\quad\quad + C[\Psi]_{k,\alpha}\left (\|u\|_{0}^{k-1}[u]_{k}\right )^{\frac{k+\alpha}{k}}\label{e:chain0}\, ,\\
\left[\Psi\circ u\right]_{k,\alpha} &\leq C\left([u]_{k,\alpha}[\Psi]_1+[u]_1^{k+\alpha}[\Psi]_{k+\alpha}\right)\label{e:chain1}\, .
\end{align} 
\end{lemma}

\subsection{Mollification estimates}

In the proof of the inductive proposition in the next section we will regularize maps $f\in C^{k,\alpha}(\bar \Omega,\R^{N})$ by convolution with a standard mollifier, i.e., a radially symmetric $\varphi_\ell\in C^{\infty}_c(B_\ell(0))$ with $\int\varphi_\ell = 1$, where $\ell>0$ denotes the length-scale. The resulting regularized maps will then have a smaller domain of definition. To counteract this, we first extend $f$ to a map $\bar f\in C^{k,\alpha}\left (\R^{2},\R^{N}\right )$ such that
\[ \|\bar f \|_{C^{k,\alpha}(\R^{2})} \leq C\|f\|_{C^{k,\alpha}(\bar \Omega)}\,,\]
where the constant $C>0$ depends only on $k,\alpha, N$ and $\Omega$. Such a procedure is given by Whitney's extension theorem, see \cite{Whitney}. We then mollify the resulting extensions at some length-scale $\ell>0$ to obtain $\tilde f= \bar f\ast \varphi_\ell \in C^{\infty}(\bar \Omega,\R^{m})$. We will not further specify this  extension procedure from now on. Such a regularization of H\"older functions enjoys the following estimates (for a proof, see \cite{CDS},\cite{DIS}).
\begin{lemma}\label{l:mollification}
For any $p, d\geq0,$ and $0<\alpha\leq1,$ we have
\begin{align*}
&[f*\varphi_\ell]_{p+d}\leq C\ell^{-d}[f]_p,\\
&\|f-f*\varphi_\ell\|_p\leq C\ell^{2-p}[f]_2, \text{ if } 0\leq p\leq2 ,\\
&\|(fg)*\varphi_\ell-(f*\varphi_\ell)(g*\varphi_\ell)\|_p\leq C\ell^{2\alpha-p}\|f\|_\alpha\|g\|_\alpha
\end{align*}
with constant $C$ depending only on $d, p, \beta, \varphi.$
\end{lemma}
\subsection{Notations} As usual, $C$ will denote constants whose values can change from line to line. If we wish to emphasize the dependents of the constant on various parameters, say $k,\alpha$, we  write $C(k,\alpha)$. Symmetric $2\times 2$ matrices are denoted by $\S$, the $2\times 2$ identity matrix by $\id$. For any matrix $M$, $M_{ij}$ denotes the $ij$-th entry of $M$. For $\xi=(\xi_1,\xi_2),\,\eta=(\eta_1,\eta_2)\in \R^{2}$, $\xi\otimes \eta$ denotes the $2\times 2$ matrix with entries $(\xi\otimes \eta)_{ij} = \xi_i\eta_j$ and we abbreviate $\xi\odot \eta=\frac12(\xi\otimes \eta + \eta\otimes \xi)=\sym{\xi\otimes\eta}$.

\section{Stage Proposition}\label{s:stage}
 
 As explained in the introduction we want to construct solutions to \eqref{e:cI} by iteratively constructing a sequence of subsolutions $(v_q,w_q)$ converging to an actual solution. In this section we state and prove the main building block of the iteration, the "stage proposition". 
 
 For a given $\theta<\frac{1}{3}$ as in Theorem \ref{t:main} we choose $b>1,\,c>\frac{3}{2}$ such that
 \begin{equation}\label{e:bcbetacondition}
 	c\geq 2b+\frac1{2b}-1,\quad \frac{1}{2bc}<\theta\,.
 \end{equation}
Define a sequence of amplitudes 
 \begin{equation}\label{d:deltaq} 
 	\delta_{q} = a^{-b^{q}},\quad q\in \N\end{equation}
 and a sequence of frequencies of the oscillations
 \begin{equation}\label{d:lambdaq}
 \lambda_q=a^{cb^{q+1}}, \quad q\in \N,
 \end{equation}
 where $a\geq a_0$ is a large constant to be chosen. We initially take sufficiently large $a_0$ such that
 \begin{equation}\label{e:ini-a}
 	\delta_1<1<\lambda_q,\quad \delta_{q+1}\leq\frac12\delta_q\quad \lambda_{q+1}\geq2\lambda_q\,.
 	\end{equation}
We then have the following 
\begin{proposition}\label{p:stage} Let $\theta <\frac{1}{3}$ and $b>1,c>\frac{3}{2}$ as in \eqref{e:bcbetacondition}, and let $\gamma\geq \frac{2}{3}$. There exists $\alpha_0=\alpha_0(b,c)\in \,]0,\gamma[$ such that for any $\alpha\in \,]0,\alpha_0[$, $D\in C^{0,\gamma}(\bar \Omega,\S)$ and $B\geq 1$ there exist constants $\tau_0(\alpha,b,c,\Omega)\in \,]0,1[$, $C_0(\alpha,b,c,\Omega, B 
)\geq 1$  and $a_0(\alpha,b,c,\Omega,\|D\|_{0,\gamma},C_0(\alpha,b,c,\Omega, B))$
with the following property.

Let $v_q\in C^{\infty}(\bar \Omega), w_q\in C^{\infty}(\bar \Omega,\R^{2})$ satisfy 
\begin{align}
&\|D-\frac{1}{2}\nabla v_q\otimes\nabla v_q -\mathrm{sym}\nabla w_q -\delta_{q+1}\id \|_{0,\alpha}\leq \tau_0\delta_{q+1}, \label{e:metricerrorq}\\
&\|v_q\|_{1}+\|w_q\|_1 \leq B-\delta_q^{\sfrac{1}{2}} \label{e:firstderq}\,, \\
&\|v_q\|_{2}+\|w_q\|_2 \leq C_0 \delta_q^{\sfrac{1}{2}}\lambda_q \label{e:secondderq}\,,
\end{align}
where $\lambda_q,\delta_q$ are defined as in  \eqref{d:deltaq}, \eqref{d:lambdaq} respectively. 
If
\begin{equation}\label{e:acondition}  a\geq a_0(\alpha,b,c,\Omega,\|D\|_{0,\gamma},C_0),
\end{equation}
then there exist $v_{q+1}\in C^{\infty}(\bar \Omega), w_{q+1}\in C^{\infty}(\bar \Omega,\R^{2})$ such that \eqref{e:metricerrorq}-\eqref{e:secondderq} hold with $q$ replaced by $q+1$ and additionally 
\begin{align}
 \|v_{q}-v_{q+1}\|_0 &\leq \delta_{q+1}^{\sfrac{1}{2}}\lambda_q^{-1}\,, \label{e:0estimate}\\
 \|v_{q}-v_{q+1}\|_1 +\|w_q-w_{q+1}\|_1 &\leq  C_0\delta_{q+1}^{\sfrac{1}{2}}\label{e:1estimate}\,.
\end{align}
\end{proposition} 

\begin{proof} We divide the proof into several steps.
	
\emph{Step 1: Parameters.} 
For a fixed $\alpha\in\,]0,\frac12[$ to be chosen below in Lemma \ref{l:parameters} we define a mollification parameter $\sigma$ by
\begin{equation}\label{d:sigma} 
	\sigma^{2-2\alpha }:=\frac{\delta_{q+1}}{\delta_q\lambda_q^{2}}
\end{equation}
and the frequency parameter of the first oscillation by
\begin{equation}\label{d:mu}
\mu := \frac{\delta_{q+2}\lambda_{q+1}^{1-2\alpha}}{\delta_{q+1}}\,.
\end{equation}
Obviously, we have
\[\sigma<1<\mu<\lambda_{q+1}.\]
Moreover,
\begin{lemma}\label{l:parameters} 
	Let $b>1, c>\frac32$ satisfy \eqref{e:bcbetacondition} and let the parameters $\sigma$ and $\mu$ be defined as in \eqref{d:sigma} and \eqref{d:mu} respectively. There exists a constant $\alpha_0(b, c)\ll1$ such that for any $0<\alpha<\alpha_0$ there exists a constant $N_0(b, c, \alpha)\in\N$  such that for any $N_0\leq N\in\N,$ it holds 
	\begin{align}
		&\sigma^{-1}\leq\mu^{1-2\alpha}\,, \label{e:p1}\\
		&\sigma^{\frac23-\alpha}\leq\delta_{q+2}\lambda_{q+1}^{-\alpha}\,,\label{e:p4} \\
		&\delta_{q+1}\left (\mu^{\alpha-1}\sigma^{-1}\right )^{N}\leq \delta_{q+2}\lambda_{q+1}^{-\alpha}\,.\label{e:p2}
	\end{align}
\end{lemma}
\begin{proof}
	Inserting the definitions of $\sigma$ and $\mu$, the estimate \eqref{e:p1} is implied by
	\[\frac{\delta_q\lambda_q^2}{\delta_{q+1}}\leq\frac{\delta_{q+2}^{(2-2\alpha)(1-2\alpha)}\lambda_{q+1}^{(1-2\alpha)^2(2-2\alpha)}}{\delta_{q+1}^{(1-2\alpha)(2-2\alpha)}}.\]
    Taking the logarithm in $a$ and dividing by $b^q\log a,$ we see that the latter inequality follows from
    \[-1+2cb+b\leq-b^2(1-2\alpha)(2-2\alpha)+cb^2(1-2\alpha)^2(2-2\alpha)+b(1-2\alpha)(2-2\alpha),\]
	which follows by 
	\begin{equation}\label{e:calpha1}
		b>\frac{1}{(1-2\alpha)^2(1-\alpha)},\quad
		c>\frac{b^2(1-2\alpha)(2-2\alpha)-b((1-2\alpha)(2-2\alpha)-1)-1}{2b(b(1-2\alpha)^2(1-\alpha)-1)}\,.
	\end{equation}
    Note that in the limit $\alpha\downarrow  0 $ the above bounds read 
    \[ b>1, \quad c>\frac{1}{2b}+1\,.\]
    Thus, with \eqref{e:bcbetacondition}, we can take $\alpha_1(b, c)$ small enough to make \eqref{e:p1} hold for $0<\alpha<\alpha_1$.
    
	Inequality \eqref{e:p4} is equivalent to
	\[\frac{\delta_{q+1}^{\frac23-\alpha}}{\delta_q^{\frac23-\alpha}\lambda_q^{\frac43-2\alpha}}\leq\delta_{q+2}^{2-2\alpha}\lambda_{q+1}^{-\alpha(2-2\alpha)}\,.\]
	Taking the logarithm in $a$ and dividing by $b^q\log a$ yields
	 \[-(b-1)(\frac23-\alpha)-2cb(\frac23-\alpha)\leq-2b^2(1-\alpha)-2cb^2\alpha(1-\alpha),\]
	 which is ensured by
	 \begin{equation}\label{e:calpha2}
	 	\alpha+b\alpha(1-\alpha)<\frac23,\quad c>\frac{2b^2(1-\alpha)-(b-1)(\frac23-\alpha)}{2b(\frac23-\alpha-b\alpha(1-\alpha))}\,.
	 \end{equation}
	 Letting $\alpha\downarrow  0 $, we get the above bounds as
	 \[ 0<\frac23, \quad c>\frac32b+\frac1{2b}-\frac12\,.\]
	 Thus, noting \eqref{e:bcbetacondition}, we can take $\alpha_2(b, c)$ small enough to make \eqref{e:p4} hold for $0<\alpha<\alpha_2$.
	 
	It remains to show \eqref{e:p2}. By the definition of $\mu$ and $\sigma$, it is equivalent to get
	\[ \left(\frac{\delta_{q+2}^{\alpha-1}\lambda_{q+1}^{(1-2\alpha)(\alpha-1)}}{\delta_{q+1}^{\alpha-1}}
	\frac{\delta_q^{\frac1{2-2\alpha}}\lambda_q^{\frac1{1-\alpha}}}{\delta_{q+1}^{\frac1{2-2\alpha}}}
	\right)^N\leq\frac{\delta_{q+2}\lambda_{q+1}^{-\alpha}}{\delta_{q+1}} \] 
	We also take the logarithm in $a$, divide $b^q\log a,$ and then get
	\begin{align*}
		&N[b^2(1-\alpha)-cb^2(1-2\alpha)(1-\alpha)-\frac{1}{2-2\alpha}+\frac{cb}{1-\alpha}-b(1-\alpha)+\frac{b}{2-2\alpha}]\\
		&\leq-b^2-\alpha cb^2+b,
	\end{align*}
which will be implied by
\begin{equation}\label{e:calpha3}
		b>\frac{1}{(1-\alpha)^2(1-2\alpha)},\quad
		c>\frac{(b-1)(1+2b(1-\alpha)^2)}{2b(b(1-\alpha)^2(1-2\alpha)-1)},
\end{equation}
and
\begin{equation}\label{e:N}
		N\geq N_0:=1+\frac{2(1-\alpha)(b^2-b+\alpha cb^2)}{2cb[b(1-\alpha)^2(1-2\alpha)-1]-(b-1)(1+2b(1-\alpha)^2)}\,.
	\end{equation}
Note that letting $\alpha\downarrow0,$ the bounds in \eqref{e:calpha3} tend to
\[b>1,\quad c>1+\frac1{2b}.\]
Hence there exists $\alpha_3(b, c)$ small enough to make \eqref{e:calpha3} hold for $0<\alpha<\alpha_3$. Hence for any $0<\alpha<\alpha_3,$ $N_0$ makes sense and then for any $N\geq N_0$, \eqref{e:p2} hold.

Finally, taking $\alpha_0=\min\{\alpha_1,\,\alpha_2,\,\alpha_3\}$ and $N_0$ as in \eqref{e:N} completes the proof.
\end{proof}

\emph{Step 2:  Mollification.} Let $\{\varphi_{\sigma}\}_{\sigma>0}$ be a standard radially symmetric mollification kernel, and set $v= v_q\ast \varphi_\sigma$, $w= w_q\ast \varphi_\sigma$ for $\sigma$ defined in \eqref{d:sigma} with $\alpha<\alpha_0$. From Lemma \ref{l:mollification} we find for $l=0,1,2$ 
\begin{align}\label{e:v_q-v}
\|v_q-v\|_{l}+\|w_q-w\|_{l} &\leq C\sigma^{2-l}\left (\|v_q\|_{2}+\|w_q\|_2\right )\leq CC_0\sigma^{2-l}\sigma^{-1+\alpha}\delta_{q+1}^{\sfrac12}\\
&\leq \delta_{q+1}^{\sfrac{1}{2}}\sigma^{1-l+\sfrac{\alpha}{2}}\,, \end{align}
and 
\begin{equation}\label{e:vw2}
\|v\|_2+\|w\|_2 \leq \|v_q\|_2+\|w_q\|_2\leq  C_0\delta_q^{\sfrac{1}{2}}\lambda_q \leq \delta_{q+1}^{\sfrac{1}{2}} \sigma^{-1+\sfrac{\alpha}{2}}
\end{equation}
if $a\geq a_0(b, c,\alpha, C_0)$ is large enough.  More generally, for any $l\geq 2$ and $\beta\in [0,1[$ it holds 
 \begin{equation}\label{e:vwestimates} \|v\|_{l,\beta}+\|w\|_{l,\beta}\leq C(l,\beta)\sigma^{2-(l+\beta) }\left (\|v_q\|_{2}+\|w_q\|_{2} \right )  \leq C(l,\beta) \delta_{q+1}^{\sfrac{1}{2}}\sigma^{-l+1-\beta+\sfrac{\alpha}{2}}\,.\end{equation}
From \eqref{e:firstderq} we moreover get
\begin{equation}\label{e:nablav} \|\nabla v\|_0 \leq \|\nabla v_q\|_0\leq B.
\end{equation}

\emph{Step 3: Decomposition.}
Now define 
\begin{equation}\label{d:h}
 H =\frac{D\ast\varphi_\sigma -\frac{1}{2}\nabla v\otimes \nabla  v -\mathrm{sym}\nabla w }{\delta_{q+1}} -\frac{\delta_{q+2}}{\delta_{q+1}}\id \,.
\end{equation}
This is the rescaled defect we want to decompose. As explained in the introduction however, we want incorporate some of the error terms resulting from the first perturbation (see \eqref{e:tildev}--\eqref{e:tildew}) in the decomposition. The error terms we want to absorb have the form 
\begin{align}
E_1(x,a,M)&:= \frac{\gamma_1(\mu x_1)}{\mu}\sqrt{a} (M-M_{22}e_2\otimes e_2)  \label{d:E1}\\
E_2(x,a)&:= -\frac{\gamma_3(\mu x_1)}{\mu}\nabla a \odot e_1 \label{d:E2} \\
E_3(x,a) &:= -\frac{\gamma_1^{2}(\mu x_1)}{2\mu^{2}} \left (\nabla\sqrt{a}\otimes \nabla \sqrt {a} - \left (\partial_2 \sqrt{a}\right )^{2}e_2\otimes e_2\right )\,.\label{d:E3} \end{align} 
where $\mu\geq 1$ is defined in  \eqref{d:mu}, $M\in C^{\infty}(\bar \Omega, \S)$,  $a\in C^{\infty}(\Omega,]0,+\infty[)$ will be the coefficient of the decomposition, and 
\begin{equation}\label{d:gammas}\gamma_1(t) = \frac{1}{\pi}\sin(2\pi t), \quad\gamma_2(t) =-\frac{1}{4\pi}\sin(4\pi t),
\quad \gamma_3 = \frac{1}{2}\gamma_1\dot{\gamma_1} + \gamma_2\,.\end{equation}
The functions $\gamma_1,\gamma_2$ are responsible for the oscillatory nature of the perturbations we will add (see \eqref{e:tildev},\eqref{e:tildew}) and are chosen as in \cite{LewickaPakzad}, \cite{CaoSz}. They satisfy the crucial identity $\frac{1}{2}\dot{\gamma_1}^{2}+\dot{\gamma_2} = 1$ which is used below. 
In the following we think of $\mu$ and $M$ as given and $a$ as variable. We thus abbreviate 
\begin{equation} E(a) := E_1(x,a,M)+E_2(x,a)+E_3(x,a)\label{d:E}\end{equation} 
Then we have the following decomposition lemma. 
\begin{lemma}\label{l:decomposition} For every $N\in\N$ and $\alpha \in \,]0,1[$ there exist constants $\tau_*(N,\alpha,\Omega)\in\, ]0,1[$ and $C(N,\alpha,\Omega)>0$ such that the following holds for any $\sigma\in \,]0,1[$. If $\mu \geq  1$, and $H,M\in C^{\infty}(\bar \Omega,\S)$ are such that 
\begin{align}  
	&\|H-\id\|_\alpha + \mu^{\alpha -1}\sigma^{-1} \leq \tau_*\,,\label{e:smallnesscondition}\\
&\|H\|_{l,\alpha}\leq C(l,\alpha)\sigma^{-l}\text{ for any }l\in \N, \label{e:hestimates}\\
&\|M\|_{l,\beta}\leq C(l,\beta)\sigma^{-l-1-\beta}\text{ for any } l\in \N, \beta\in [0,1[\label{e:Mestimates}
\end{align} 
for some constants $C(l,\alpha), C(l,\beta)\geq 1$ then there exist maps $\{a^{(k)}\}_{k=0}^{N}\subset C^{\infty}(\bar \Omega,]0,+\infty[)$, $\{w^{(k)}\}_{k=0}^{N} \subset C^{\infty}(\bar \Omega, \R^{2})$ such that for $k=1,\ldots,N$ 
\begin{equation}\label{e:Kaellendecomp} a^{(k)} \id = H -\mathrm{sym} \nabla w^{(k)}+E(a^{(k-1)})\,.
\end{equation}
 with estimates
\begin{align}
\| a^{(k)}-1\|_{0,\alpha}+ \|w^{(k)}\|_{1,\alpha}&\leq \frac{1}{2} \label{e:decompestimate0}\\
\|a^{(k)} - 1 \|_{l,\alpha} + \|w^{(k)}\|_{l+1,\alpha} &\leq C(l,\alpha,\Omega)\sigma^{-1}\mu ^{l-1} \quad \text{ for } 1\leq l \leq N+3-k \label{e:decompestimate1}\\
\|E(a^{(k)})-E(a^{(k-1)}) \|_{0,\alpha}&\leq C(N,\alpha,\Omega)\left (\mu^{\alpha-1}\sigma^{-1}\right )^{k+1}\,.\label{e:differrors}
\end{align}
\end{lemma} 
\begin{remark}
Due to the appearance of $\nabla a^{(k)}$ in $E(a^{(k)})$ one cannot use a fixed point theorem to get the limit of $a^{(k)}$. Therefore, we  only use finitely many times iteration \eqref{e:Kaellendecomp}.
\end{remark}
\begin{proof}
	The proof is in Section \ref{s:decom}. 
\end{proof}
To apply the decomposition lemma, we fix $\alpha<\alpha_0$  and $N\geq N_0$  such that Lemma \ref{l:parameters} holds.
With such $\alpha$, after taking $a$ large, we get the following.
\begin{lemma}\label{l:h} 
	If $\tau_0<\frac{\tau_*}{6}$ and $a\geq a_0(\alpha,b,c, \Omega)$ are chosen large enough, then it holds
\begin{align}
&\|H-\id\|_{0,\alpha} \leq \frac{\tau_*}{2}\label{e:h0estimate} \\
&\|H-\id\|_{l,\alpha}\leq C(l,\alpha)\sigma^{-l}, \text{ for } 1\leq l\in \N.\label{e:h1estimate} 
\end{align} 
\end{lemma}
\begin{proof}
Thanks to \eqref{e:metricerrorq} and \eqref{e:secondderq}, we can estimate with Lemma \ref{l:mollification}
\begin{align*}
\|H-\id\|_{0,\alpha}&\leq \frac{1}{\delta_{q+1}}\|(D-\frac{1}{2}\nabla v_q\otimes \nabla v_q -\sd w_q-\delta_{q+1}\id)\ast \varphi_\sigma\|_{0,\alpha} + C\frac{\delta_{q+2}}{\delta_{q+1}} \\
&\quad \quad +\frac{1}{2\delta_{q+1}}\|\nabla v\otimes \nabla v - \left (\nabla v_q\otimes \nabla v_q\right )\ast \varphi_\sigma\|_{0,\alpha}\\
&\leq 2\tau_0+ \frac{C}{\delta_{q+1}} \sigma^{2-\alpha}\|v_q\|_{2}^{2} \leq 3\tau_0 \,,
\end{align*} if  $a\geq a_0(\tau_0,C_0)$ is chosen large enough. Moreover, using \eqref{e:Leibniz} and Lemma \ref{l:mollification}  we infer
\begin{align*} \|H-\id\|_{l,\alpha}&\leq \frac{C(l,\alpha)}{\delta_{q+1}}\left (\tau_0 \delta_{q+1}\sigma^{-l}+ \|v_q\|_{2}^{2}\sigma^{2-(l+\alpha)}\right )\leq C(l,\alpha)\sigma^{-l}\,. \qedhere
\end{align*}
\end{proof} 
Moreover, from Lemma \ref{l:parameters}, it is easy to get
\begin{equation}\label{e:tau1}
	\mu^{\alpha-1}\sigma^{-1}\leq\mu^{-\alpha}<\frac{\tau_*}{2}
	\end{equation}
provided that $a$ is sufficiently large depending on $\tau_*$, i.e. on $b,c,\alpha,\Omega$. Set
\begin{equation}\label{d:M}
M:= \delta_{q+1}^{-\sfrac{1}{2}} D^{2}v\,,\end{equation}
for which we have (thanks to \eqref{e:vwestimates})
\begin{equation}\label{e:M}
 \| M \|_{l,\beta } \leq  C(l,\beta) \sigma^{-l-1-\beta+{\alpha}}\leq C(l,\beta) \sigma^{-l-1-\beta}\end{equation}
for $l\in \N$, $\beta\in [0,1[$. 
In view of \eqref{e:h0estimate}--\eqref{e:h1estimate}, \eqref{e:tau1} and \eqref{e:M}, 
we can therefore apply Lemma \ref{l:decomposition} to get maps $\{a^{(k)}\}_{k=0}^{N}, \{w^{(k)}\}_{k=0}^{N}$ satisfying the conclusions of the lemma. In particular,  
\begin{equation}\label{e:Ndecom}
	a^{(N)} \id = H -\mathrm{sym} \nabla w^{(N)}+E(a^{(N-1)})=a^{(N)}(e_1\otimes e_1+e_2\otimes e_2)\,
\end{equation}
with $M$ of $E_1$ defined in \eqref{d:M}.

\emph{Step 4: First perturbation.} Set
\begin{equation}\label{d:aw}
 r= \delta_{q+1}a^{(N)}\,,\quad w_c = \delta_{q+1}w^{(N)}
\end{equation}
and define the first perturbation 
\begin{align} 
&\tilde v= v+\frac{\sqrt{r}}{\mu}\gamma_1(\mu x_1) \label{e:tildev}\\
&\tilde w= w+w_c -\frac{\sqrt{r}}{\mu}\gamma_1(\mu x_1)\nabla v +\frac{r}{\mu}\gamma_2(\mu x_1) e_1 \,.\label{e:tildew} 
\end{align}
A quick computation gives 
\begin{equation}\label{e:nablavtilde} \nabla\tilde v= \nabla v +\sqrt{r}\dot{\gamma_1}e_1 + \frac{\gamma_1}{\mu}\nabla\sqrt{r}\end{equation}
and 
\begin{align}
\nabla \tilde w & = \nabla w + \nabla w_c -\sqrt{r}\dot{\gamma_1} \nabla v \otimes e_1 +r \dot{\gamma_2} e_1\otimes e_1 \nonumber\\
 &\quad \quad - \frac{\gamma_1}{\mu}\nabla v \otimes \nabla\sqrt{r} -\frac{\sqrt{r}}{\mu}\gamma_1 D^{2}v +\frac{\gamma_2}{\mu} e_1\otimes \nabla r
  \,.
\end{align}
Consequently 
\begin{align*} 
\frac{1}{2}\nabla \tilde v\otimes \nabla\tilde v &= \frac{1}{2}\nabla v \otimes \nabla v + \frac{1}{2}r \dot{\gamma_1}^{2} e_1\otimes e_1 + \sqrt{r}\dot{\gamma_1} \nabla v \odot e_1 \\ 
&\quad\quad +\frac{\gamma_1}{\mu}\nabla\sqrt{r}\odot \nabla v +\frac{\gamma_1\dot{\gamma_1}}{\mu} \sqrt{r}e_1\odot\nabla\sqrt{r}\\
&\quad\quad +\frac{\gamma_1^{2}}{2\mu^{2}}\nabla\sqrt{r}\otimes \nabla \sqrt{r}\,,
\end{align*}
and therefore  
\begin{align*}
\frac{1}{2}\nabla \tilde v\otimes \nabla\tilde v +\mathrm{sym}\nabla \tilde w& = \frac{1}{2}\nabla v\otimes \nabla v +\mathrm{sym}\nabla  w + \mathrm{sym}\nabla w_c +re_1\otimes e_1 \\ 
&\quad +\frac{\gamma_3}{\mu}\nabla r\odot e_1  - \frac{\sqrt{r}}{\mu}\gamma_1 D^{2}v +\frac{\gamma_1^{2}}{2\mu^{2}	}\nabla\sqrt{r}\otimes \nabla \sqrt{r}
\end{align*} 
recalling the definition of $\gamma_3 $ in \eqref{d:gammas} and $\frac{1}{2}\dot{\gamma_1}^{2}+\dot{\gamma_2} \equiv 1$. But from \eqref{e:Kaellendecomp} we have
\[ \mathrm{sym}\nabla w_c =\delta_{q+1}H -r \id + \delta_{q+1}E(a^{(N-1)})\] 
so that 
\begin{align}
\frac{1}{2}\nabla \tilde v\otimes \nabla\tilde v +\mathrm{sym}\nabla \tilde w
& = \frac{1}{2}\nabla v\otimes \nabla v +\mathrm{sym}\nabla  w +\delta_{q+1}H \nonumber\\
&\quad \quad+\delta_{q+1}\left (E(a^{(N-1)})-E(a^{(N)}))\right )\nonumber \\ 
&\quad \quad  + \left (\frac{\sqrt{r}}{\mu}\gamma_1\partial^{2}_{22} v -\frac{\gamma_1^{2}}{2\mu^{2}}\left (\partial_2 \sqrt{r}\right )^{2}-r\right ) e_2\otimes e_2 \nonumber\\
& = D\ast\varphi_\sigma -\delta_{q+2}\id + \delta_{q+1}\left (E(a^{(N-1)})-E(a^{(N)})\right ) \nonumber\\
& \quad \quad + \left (\frac{\sqrt{r}}{\mu}\gamma_1\partial^{2}_{22} v -\frac{\gamma_1^{2}}{2\mu^{2}}\left (\partial_2 \sqrt{r}\right )^{2}-r\right ) e_2\otimes e_2
\end{align} 

\emph{Step 5: Second perturbation.} For the second perturbation we set 
\begin{equation}
s = r -\frac{\sqrt{r}}{\mu}\gamma_1\partial^{2}_{22} v +\frac{\gamma_1^{2}}{2\mu^{2}}\left (\partial_2 \sqrt{r}\right )^{2}
\end{equation} 
and let 
\begin{align}
&v_{q+1}=\tilde v +\frac{\sqrt{s}}{\lambda_{q+1}} \gamma_1\left (\lambda_{q+1}x_2\right ) \\ 
&w_{q+1} = \tilde w -\frac{\sqrt{s}}{\lambda_{q+1}}\gamma_1\left (\lambda_{q+1}x_2\right )\nabla \tilde v +\frac{s}{\lambda_{q+1}}\gamma_2\left (\lambda_{q+1}x_2\right ) e_2\,.
\end{align}
We will show below in \eqref{e:b} that these choices are well-defined. As above we find 
\begin{align}
 \frac{1}{2}\nabla v_{q+1}\otimes \nabla v_{q+1} +\mathrm{sym}\nabla w_{q+1}\nonumber
=& \frac{1}{2}\nabla \tilde v\otimes \nabla\tilde v +\mathrm{sym}\nabla \tilde w +s e_2\otimes e_2  +\frac{\gamma_3}{\lambda_{q+1}}\nabla s\odot e_2\nonumber\\ 
&\quad - \frac{\sqrt{s}}{\lambda_{q+1}}\gamma_1 D^{2}\tilde v +\frac{\gamma_1^{2}}{2\lambda_{q+1}^{2}	}\nabla\sqrt{s}\otimes \nabla \sqrt{s}\nonumber\\
 =&D\ast\varphi_{\sigma} - \delta_{q+2 }\id +\delta_{q+1}\left (E(a^{(N-1)})-E(a^{(N)})\right )\nonumber\\ 
&\quad +\frac{\gamma_3}{\lambda_{q+1}}\nabla s\odot e_2  - \frac{\sqrt{s}}{\lambda_{q+1}}\gamma_1 D^{2}\tilde v \nonumber \\
&\quad +\frac{\gamma_1^{2}}{2\lambda_{q+1}^{2}	}\nabla\sqrt{s}\otimes \nabla \sqrt{s}\,.\label{e:errorq+1}
\end{align}
In particular,
\begin{align}
	 &D-\frac{1}{2}\nabla v_{q+1}\otimes \nabla v_{q+1} -\mathrm{sym}\nabla w_{q+1} -\delta_{q+2}\id\\
	  =& D-D\ast\varphi_\sigma -
\delta_{q+1}\left (E(a^{(N-1)})-E(a^{(N)}\right )\nonumber\\
&\quad -\frac{\gamma_3}{\lambda_{q+1}}\nabla s\odot e_2  - \frac{\sqrt{s}}{\lambda_{q+1}}\gamma_1 D^{2}\tilde v \nonumber \\ 
&\quad +\frac{\gamma_1^{2}}{2\lambda_{q+1}^{2}	}\nabla\sqrt{s}\otimes \nabla \sqrt{s}\,,\label{e:errors}
\end{align}

\bigskip

It now remains to show that the $\alpha$-H\"older norm of the right hand side is less than $\sigma_0\delta_{q+2}$, that \eqref{e:secondderq} holds for $q+1$, and that \eqref{e:0estimate}--\eqref{e:1estimate} hold provided $a$ and $C_0$ are chosen large enough and $\alpha$ small enough. In the following estimates, the constants $C$ may depend on $\alpha,b,c,N,\Omega$ and $B$, 
\emph{but not on} $C_0$ or $a$. 
%

\emph{Step 6: Estimates on $r,\tilde v, \tilde w$.} 
Using \eqref{e:decompestimate0} and \eqref{e:decompestimate1} we find 
\[ \|r-\delta_{q+1}\|_{0,\alpha}\leq \frac{1}{2}\delta_{q+1}\,,\quad \|r-\delta_{q+1}\|_{l,\alpha}\leq C\delta_{q+1}\mu^{l-1}\sigma^{-1}\,\text{ for } l=1,2,3\,\]
and with the help of Lemma \ref{l:chain}
\[ \|\sqrt{r}-\delta_{q+1}^{\sfrac{1}{2}}\|_{0,\alpha}\leq \frac{1}{2}\delta_{q+1}^{\sfrac{1}{2}}\,,\quad \|\sqrt{r}-\delta_{q+1}^{\sfrac{1}{2}} \|_{l,\alpha}\leq C\delta_{q+1}^{\sfrac{1}{2}}\mu^{l-1}\sigma^{-1}\, \text{ for } l=1,2,3\,.\]
From there we get \[\|\tilde v -v \|_{0}\leq \frac{3}{2}\delta_{q+1}^{\sfrac{1}{2}}\mu^{-1}\,,\] and 
\begin{equation}\label{e:tildevestimates}\|\tilde v -v \|_{l} \leq C \delta_{q+1}^{\sfrac{1}{2}}\left ( \mu^{l-2}\sigma^{-1} + \mu^{l-1} \right ) \leq C\delta_{q+1}^{\sfrac{1}{2}}\mu^{l-1}  \, \text{ for } l=1,2, \end{equation}
since $\mu^{\alpha-1}\sigma^{-1}<1$ by \eqref{e:tau1}.
For $\tilde w$ we estimate
for $l=1,2$
\begin{align*} \|\frac{\sqrt{r}}{\mu}\gamma_1(\mu x_1)\nabla v\|_l&\leq C \left (\mu^{-1}\|\sqrt{r}\|_{l,\alpha} + \|\sqrt{r}\|_{0,\alpha}\left (\mu^{l-1} + \mu^{-1}\|v\|_{l+1}\right )\right ) \\
&\leq C\delta_{q+1}^{\sfrac{1}{2}}\left (\mu^{l-2}\sigma^{-1} +\mu^{l-1} + \mu^{-1}\sigma^{-l+{\alpha}}\right )\leq C\delta_{q+1}^{\sfrac{1}{2}} \mu^{l-1} \,,
\end{align*}
again due to $\mu^{\alpha-1}\sigma^{-1}<1$. 
With \eqref{e:nablavtilde},\eqref{e:vwestimates}, we have
\[\|\frac{r}{\mu}\gamma_2(\mu x_1) e_1\|_l \leq C\mu^{-1}\delta_{q+1}\left (\mu^{l-1}\sigma^{-1} + \mu^{l}\right )\leq C\delta_{q+1}\mu^{l-1} \]
for $l=1,2$, then in combination with \eqref{e:decompestimate1},
\begin{equation}\label{e:tildewestimates}
\|\tilde w - w\|_{l} \leq C\delta_{q+1}\sigma^{-1}\mu^{l-2}+ C\delta_{q+1}^{\sfrac{1}{2}}\mu^{l-1} \leq C\delta_{q+1}^{\sfrac{1}{2}}\mu^{l-1}\,\end{equation}
for $l=1,2$.

\emph{Step 7: Estimates on $s$, $v_{q+1}$, $w_{q+1}$.}
By \eqref{e:vwestimates} and $\sigma^{-\alpha}\leq\mu^{\alpha}$  
we can estimate
\begin{align*}\|\frac{s}{\delta_{q+1}} -1 \|_{0,\alpha} &\leq \|a^{(N)} -1\|_{0,\alpha} + \delta_{q+1}^{-\sfrac{1}{2}}\left (\mu^{\alpha-1}\|v\|_{2}+\mu^{-1}\|v\|_{2,\alpha}\right ) +\mu^{\alpha-2}\|\sqrt{a^{(N)}}\|_{1,\alpha}^{2} \\
&\leq \frac{1}{2}+C\left (\mu^{\alpha-1}\sigma^{{\alpha}-1} +\mu^{-1}\sigma^{-1-{\alpha}}\right ) +C\mu^{\alpha-2}\sigma^{-2}\\
&\leq \frac{1}{2} + C\mu^{-\alpha}\leq \frac{3}{4} \end{align*}
 if  $a$ is large enough. It follows that 
 \begin{equation}\label{e:b} \|s-\delta_{q+1}\|_{0,\alpha}\leq \frac{3}{4}\delta_{q+1} \text{ and } \|\sqrt{s}	-\delta_{q+1}^{\sfrac{1}{2}}\|_{0,\alpha}\leq \frac{3}{4}\delta_{q+1}^{\sfrac{1}{2}}\,,\end{equation}
and hence $v_{q+1}$ and $w_{q+1}$ are well-defined. For the derivatives we find for $l=1,2$
\begin{align*}
\|\frac{s}{\delta_{q+1}}-1\|_{l,\alpha}& \leq \|a^{(N)}-1\|_{l,\alpha}+ C\delta_{q+1}^{-\sfrac{1}{2}}\mu^{-1}\left ( \left (\|\sqrt{a^{(N)}}\|_{l,\alpha} +\mu^{l+\alpha}\right )\|v\|_2 + \|v\|_{l+2,\alpha}\right ) \\
& \quad +C\left ( \mu^{l+\alpha-2}\|\sqrt{a^{(N)}}\|_{1,\alpha}^{2} +\mu^{-2}\|\sqrt{a^{(N)}}\|_{l+1,\alpha}\|\sqrt{a^{(N)}}\|_{1,\alpha}\right )\\
&\leq C\mu^{l-1}\sigma^{-1} +C\delta_{q+1}^{-\sfrac{1}{2}}\mu^{-1}\left (\mu^{l-1}\sigma^{-1}+\mu^{l+\alpha}\right )\delta_{q+1}^{\sfrac{1}{2}}\sigma^{{\alpha}-1} \\
&	\quad +C\delta_{q+1}^{-\sfrac{1}{2}}\mu^{-1}\sigma^{-(l+\alpha)}\cdot\delta_{q+1}^{\sfrac{1}{2}}\sigma^{{\alpha}-1}+
C\left(\mu^{l+\alpha-2}\sigma^{-2}+\mu^{l-2}\sigma^{-2}\right )\\
&\leq C\mu^{l-1+\alpha}\sigma^{-1} \leq C\mu^{l-\alpha}\leq \mu^{l}\,,
\end{align*}
if $a$ is large enough by \eqref{e:p1}. Thus
\begin{equation}\label{e:bestimates}
\|s-\delta_{q+1}\|_{l,\alpha}\leq \delta_{q+1}\mu^{l}\,,\quad  \|\sqrt{s}-\delta_{q+1}^{\sfrac{1}{2}}\|_{l,\alpha}\leq \delta_{q+1}^{\sfrac{1}{2}}\mu^{l}\,
\end{equation}
for $l=1,2$. This yields the estimates 
\begin{equation}\label{e:vq+10}
\|v_{q+1}-\tilde v\|_0\leq \frac{3}{2}\delta_{q+1}^{\sfrac{1}{2}}\lambda_{q+1}^{-1}\,,
\end{equation} and 
\begin{equation}\label{e:vq+1estimates}
\|v_{q+1}-\tilde v\|_{l} \leq C\delta _{q+1}^{\sfrac{1}{2}}\lambda_{q+1}^{-1}\left (\lambda_{q+1}^{l}+\mu^{l}\right )\leq C\delta_{q+1}^{\sfrac{1}{2}}\lambda_{q+1}^{l-1}\,
\end{equation}
for $l=1,2$, where we have used $
\lambda_{q+1} > \mu.$
Concerning the estimate for $w_{q+1}$ we have for $l=1,2$
\begin{align}
\|w_{q+1}-\tilde w\|_{l} &\leq C\delta_{q+1}^{\sfrac{1}{2}}\lambda_{q+1}^{-1}\left (\lambda_{q+1}^{l}+ \mu^{l}+\delta_{q+1}^{\sfrac{1}{2}}\mu^{l}+\delta_{q+1}^{\sfrac{1}{2}}\sigma^{-l+\sfrac{\alpha}{2}}\right ) \nonumber\\
&\quad + C\delta_{q+1}\lambda_{q+1}^{-1}\left (\lambda_{q+1}^{l}+\mu^{l}\right ) \nonumber \\
&\leq C\delta_{q+1}^{\sfrac{1}{2}}\lambda_{q+1}^{l-1}\label{e:wq+1estimates}\,. \end{align}
Combining with \eqref{e:vw2} and \eqref{e:tildevestimates},  \eqref{e:tildewestimates} this yields
\[ \|v_{q+1}\|_2 +\|w_{q+1}\|_{2} \leq C\delta_{q+1}^{\sfrac{1}{2}}\left (\sigma^{-1+\sfrac{\alpha}{2}} +\mu+\lambda_{q+1}\right ) \leq C\delta_{q+1}^{\sfrac{1}{2}}\lambda_{q+1}\,.\]
Moreover,  
\begin{align*}
	\|v_q-v_{q+1}\|_0 &\leq \|v_q-v\|_0+\|v-\tilde v\|_0+\|\tilde v-v_{q+1}\|_0\\
	&\leq\delta_{q+1}^{\sfrac{1}{2}}\left (\sigma + \frac{3}{2}\mu^{-1}+\frac{3}{2}\lambda_{q+1}^{-1}\right )\leq \delta_{q+1}^{\sfrac{1}{2}}\lambda_q^{-1}
\end{align*}
for $a$ large enough in view of \[ \sigma <\sigma^{1-\alpha}\leq\left (\frac{\delta_{q+1}}{\delta_{q}}\right )^{\sfrac{1}{2}} \lambda_q^{-1}\,.\]
Lastly, we find
\[ \|v_{q+1}-v_q\|_{1} +\|w_{q+1}-w_{q}\|_1 \leq \delta_{q+1}^{\sfrac{1}{2}} \sigma^{\sfrac{\alpha}{2}} + C\delta_{q+1}^{\sfrac{1}{2}}\leq C\delta_{q+1}^{\sfrac{1}{2}}\,.\]
This gives \eqref{e:secondderq}, \eqref{e:0estimate} and \eqref{e:1estimate} provided $C_0$ is large enough (depending on $\alpha,b,c,\Omega, B$), and also \eqref{e:firstderq}
\[\|v_{q+1}\|_1+\|w_{q+1}\|_1\leq \|v_q\|_1+\|w_q\|_1+\|v_{q+1}-v_q\|_{1} +\|w_{q+1}-w_{q}\|_1\leq B-\delta_q^{\sfrac{1}{2}}+C\delta_{q+1}^{\sfrac{1}{2}}\leq B-\delta_{q+1}^{\sfrac{1}{2}}\]
by taking $a$ large enough.

\emph{Step 8: Estimates on $\mathcal E_1,\mathcal E_2$ and $\mathcal E_3$.} 
We are left to show \eqref{e:metricerrorq} for $q+1$. By \eqref{e:errors} we have
\[ D-\frac{1}{2}\nabla v_{q+1}\otimes \nabla v_{q+1} -\mathrm{sym}\nabla w_{q+1} -\delta_{q+2}\id = \mathcal E_1+\mathcal E_2+\mathcal E_3 \]
for
\begin{align*} \mathcal E_1 &= D-D\ast\varphi_\sigma \\
\mathcal E_2 &=-\delta_{q+1}\left (E(a^{(N-1)})-E(a^{(N)})\right )\\
\mathcal E_3 & =-\frac{\gamma_3}{\lambda_{q+1}}\nabla s\odot e_2  - \frac{\sqrt{s}}{\lambda_{q+1}}\gamma_1 D^{2}\tilde v +\frac{\gamma_1^{2}}{2\lambda_{q+1}^{2}	}\nabla\sqrt{s}\otimes \nabla \sqrt{s}\,,
\end{align*}
and we estimate each term separately. First, by Lemma \ref{l:mollification}, $\sigma<1$ and $\gamma>\frac23,$
\begin{equation}\label{e:E1}
\|\mathcal E_1\|_{0,\alpha}\leq C\sigma^{\gamma-\alpha}\|D\|_{0,\gamma}\leq C\sigma^{\frac23-\alpha}\|D\|_{0,\gamma} \,, \end{equation}
which combining with \eqref{e:p4} will leads to 
\[\|\mathcal E_1\|_{0,\alpha}\leq C\delta_{q+2}\lambda_{q+1}^{-\alpha}\|D\|_{0,\gamma}\leq\frac{\tau_0}{3}\delta_{q+2},\]
if $a$ is large enough additionally depending on $\|D\|_{0,\gamma}$. Next, we find from \eqref{e:differrors}
\begin{equation}\label{e:E2}
\|\mathcal E_2\|_{0,\alpha}\leq C\delta_{q+1}\left (\mu^{\alpha-1}\sigma^{-1}\right )^{N+1}\leq\frac{\tau_0}{3}\delta_{q+2}\,,
\end{equation}
by \eqref{e:p2} and taking $a$ large enough. Lastly, using \eqref{e:bestimates} and \eqref{e:tildevestimates}, we can estimate
\[ \|\frac{\gamma_3}{\lambda_{q+1}}\nabla s\odot e_2 \|_{0,\alpha }\leq C\delta_{q+1}\lambda_{q+1}^{\alpha-1}\mu^{\alpha}\sigma^{-1}\,,\]
\[ \|\frac{\sqrt{s}}{\lambda_{q+1}}\gamma_1 D^{2}\tilde v\|_{0,\alpha} \leq C\delta_{q+1}\left (\lambda_{q+1}^{\alpha-1} \mu +\lambda_{q+1}^{-1}\mu^{1+\alpha}\right )\leq C\delta_{q+1}\lambda^{\alpha-1}\mu \,,\]
and \[\|\frac{\gamma_1^{2}}{2\lambda_{q+1}^{2}	}\nabla\sqrt{s}\otimes \nabla \sqrt{s}\|_{0,\alpha}\leq C\delta_{q+1}\lambda_{q+1}^{\alpha-2}\mu^{2\alpha}\sigma^{-2}\,.\] 
Then using $\mu^{\alpha-1}\sigma^{-1}<1$ and $\lambda_{q+1}\geq \mu $, we get
\[\|\mathcal E_3\|_{0,\alpha}\leq C \delta_{q+1}\lambda^{\alpha-1}_{q+1}\mu\leq\frac{\tau_0}{3}\delta_{q+2}\,,\]
by \eqref{d:mu} and taking $a$ sufficiently large.
\phantom\qedhere
\end{proof}

\section{Proof of Theorem \ref{t:main}}\label{s:proofoftmain}

Let $\underline v\in C^{0}(\bar \Omega)$, $f\in L^{p}(\Omega)$, $\epsilon >0$ and $\theta <\frac{1}{3}$ be given. By extending and mollifying we can assume $\underline v\in C^{\infty}(\bar \Omega)$. Then Theorem \ref{t:main} will be proven by seeking $C^{1, \theta}$ solutions to \eqref{e:MA} $C^0$ neighbourhood of $\underline v,$ which is realised by solving the system \eqref{e:cI} and \eqref{e:f} in the 

We first solve \eqref{e:f}. Let $u\in W^{2,p}(\Omega)$ be the potential theoretic solution of $-\Delta u = f$ on $\Omega$. By Morrey's embedding it holds $u\in C^{0,\gamma}(\bar \Omega)$ for $\gamma= 2-\frac{2}{p}\geq \frac{2}{3}$ and
\[\|u\|_{0,\gamma}\leq C\|f\|_{L^p}.\]
We then fix $b>1, c>\frac{3}{2}$ such that $\frac{1}{2bc}< \theta$,  some $\alpha<\alpha_0(b,c,\gamma)<\gamma$ and define
\[ A=(u+\kappa)\id,\text{ with }
\kappa=\frac{\|u\|_{0,\gamma}+\|\underline v\|^2_{1,\gamma}+100}{\tau_0}>1.\] 
where $\tau_0=\tau_0(\alpha,b,c,\Omega)<1$ is from Proposition \ref{p:stage}, then
\[\|A\|_{0,\gamma}\leq\|u\|_{0,\gamma}+\kappa,\quad 
-\curl~\curl A=f.\]

Next we shall apply Proposition \ref{p:stage} to solve \eqref{e:cI}. To this end fix
\begin{align}
B=& \|\underline v\|_1+2 \nonumber\\
	a\geq& a_0(\alpha,b,c,\Omega,\|A\|_{0,\gamma},C_0(\alpha, b, c, \Omega, B))\label{e:achoice}
\end{align}
Note that $a_0(\alpha,b,c,\Omega,\|D\|_{0,\gamma}, C_0)$ in Proposition \ref{p:stage} is
monotonically increasing in $\|D\|_{0,\gamma}$. Hence for any choice of $a$ according to \eqref{e:achoice} it follows also
\[ a\geq a_0(\alpha,b,c,\Omega,\|\delta_1\kappa^{-1}A\|_{0,\gamma},C_0(\alpha,b,c, \Omega, B)),\]
where $\delta_1=a^{-b}<1$. Therefore, setting
\[\bar A={\delta_1}{\kappa}^{-1}A,\quad v_0=\delta_1^{\sfrac{1}{2}}{\kappa}^{-\sfrac12}\underline v,\quad w_0=0\,,\]
gives $\bar A\in C^{0,\gamma}(\bar \Omega,\S), v_0\in C^{\infty}(\bar \Omega), w_0\in C^{\infty}(\bar \Omega,\R^{2})$ and 
\begin{align*} 
&\|\bar A-\frac{1}{2}\nabla v_0\otimes \nabla v_0 -\mathrm{sym} \nabla w_0 -\delta_1 \id\|_{0,\alpha} \leq \tau_0\delta_1 \\
&\|v_0\|_{1}+\|w_0\|_{1}\leq B-\delta_0^{\sfrac12},\\
&\|v_0\|_{2}+\|w_0\|_{2} \leq \delta_0^{\sfrac{1}{2}}\lambda_0.
\end{align*}
We can therefore iteratively apply Proposition \ref{p:stage} with $D=\bar A$ to find  sequences $\{v_q\}_{q\in \N}\subset C^{\infty}(\bar \Omega)$, $\{w_q\}_{q\in \N}\subset C^{\infty}(\bar \Omega,\R^{2})$ satisfying \eqref{e:metricerrorq}--\eqref{e:1estimate} for every $q$. Observe that by interpolation \eqref{e:interpolation} 
\[ \|v_{q+1}-v_q\|_{1,\theta}\leq C \|v_{q+1}-v_q\|_{1}^{1-\theta}\|v_{q+1}-v_q\|_{2}^{\theta}\leq C\delta_{q+1}^{\sfrac{1}{2}}\lambda_{q+1}^{\theta}= Ca^{-b^{q+1}\left (\sfrac{1}{2}-\theta cb\right )}\,.\]
Since $\theta<\frac{1}{2bc}$ by assumption this shows that $\{v_q\}$ is Cauchy in $C^{1,\theta}(\bar \Omega)$ and therefore converges to some $v_\infty$. Moreover, from \eqref{e:1estimate} we see that $w_q\to w_{\infty}$ in $C^{1}$. Passing to the limit in \eqref{e:metricerrorq} shows that 
\[ \bar A=D= \frac{1}{2}\nabla v_{\infty}\otimes \nabla v_{\infty} +\mathrm{sym}\nabla w_\infty\,.\] 
We now claim that if $a$ is chosen large enough (depending on $\kappa$ and $\epsilon$) it holds 
\begin{equation}\label{e:closeness0}
	 \|v_0-v_{\infty}\|_{0}< \delta_1^{\sfrac{1}{2}}{\kappa}^{-\sfrac12}\epsilon \,.
 \end{equation}
Indeed, from \eqref{e:0estimate}, we get
\[
 \|v_0-v_{\infty}\|_{0}\leq \sum_{q=0}^{\infty} \delta_{q+1}^{\sfrac{1}{2}}\lambda_q^{-1}\,. \]
Since $\delta_{q+1}\leq\delta_1$ is monotonically decreasing, \eqref{e:closeness0} is implied by 
\[  \sum_{q=0}^{\infty}\lambda_q^{-1} <\frac{\epsilon}{\sqrt{\kappa}}\,,\] which is easily achieved by choosing $a$ large enough. Setting 
\[\bar v =\delta_1^{-\sfrac{1}{2}}{\kappa}^{\sfrac12}v_{\infty}\]
we have
\[\|\bar v-\underline v\|_0=\delta_1^{-\sfrac{1}{2}}{\kappa}^{\sfrac12}\|v_0-v_\infty\|_0\leq\epsilon,\]
and 
\[\mathcal{D}\text{et}D^2\bar v=-\delta_1^{-1}\kappa~\curl~\curl\bar A=-\curl~\curl A=f,\]
then finish the proof. 

\section{Proof of Lemma \ref{l:decomposition}}\label{s:decom}
 The proof proceeds by inductive construction of $a^{(k)},w^{(k)}$ satisfying \eqref{e:Kaellendecomp}, \eqref{e:decompestimate0} and \eqref{e:decompestimate1}, and then we show that \eqref{e:differrors} holds. In order to do so we need slightly subtler estimates (see \eqref{e:decompestimate2}, \eqref{e:decompestimate3}). 
 
 First of all, $a^{(0)}\in C^{\infty}(\bar \Omega), w^{(0)}\in C^{\infty}\left (\bar \Omega,\R^{2}\right)$ satisfying \begin{equation}\label{e:easydecomp} 
H-\mathrm{sym}\nabla w^{(0)} = a^{(0)} \id 
\end{equation}
can be easily found by Cauchy transform in complex analysis, which has also been established in \cite{CaoSz} in annother way.  Indeed, fix a smooth bounded open set $\tilde \Omega$ such that $\Omega\Subset \tilde \Omega$. The Cauchy transform  
\[ \mathcal C[h](z) := -\frac{1}{\pi}\int_{\tilde \Omega}\frac{h(\omega)}{\omega - z }\,d\omega \] 
satisfies 
\[\partial_{\bar z} \mathcal C[h] = h  \text{ (where $\partial_{\bar z }= \frac{1}{2}(\partial_{x_1} +i\partial_{x_2}$))}\]
 and 
\begin{equation} \label{e:Cauchyestimate}
 \|\mathcal C[h]\|_{l+1,\alpha} \leq C(l,\alpha)\|h\|_{l,\alpha}
\end{equation}
for any $l\in\N$, $f\in C^{l,\alpha}(\tilde \Omega)$. Moreover, on $C^{l,\alpha}_c(\tilde \Omega)$, it commutes with any differential operator of order $l$ (see for example Appendix A in \cite{DIS}). Extend $H$ given in \eqref{d:h} to $H\in C^{\infty}_c(\tilde \Omega , \R^{2\times 2})$,  let $\hat H = H_{11}-H_{22}+i(H_{12}+H_{21})$ and set
\begin{align*}
 w^{(0)}&=\frac{1}{2}\mathcal C[\hat H] \\
 a^{(0)}&=  H_{22}-\partial_{2}w_2^{(0)}
\end{align*}
A quick computation then shows that \eqref{e:easydecomp} is satisfied. Moreover, using \eqref{e:Cauchyestimate} for $h= \widehat{H-\id}$ we get
\begin{equation}\label{e:a00} \|a^{(0)}-1\|_{0,\alpha } + \|w^{(0)}\|_{1,\alpha} \leq C \|H-\id\|_{0,\alpha }\leq C\tau_* <\frac{1}{4}\end{equation}
for $\tau_*$ small enough, and
\begin{equation}\label{e:a0l} \|\partial^{I}(a^{(0)}-1)\|_{0,\alpha} + \|\partial^{I}w^{(0)}\|_{1,\alpha} \leq C\|\partial^{I}(h-\id)\|_{0,\alpha}\leq 
 C(N,\alpha)\sigma^{-|I| }\end{equation}
 for all multi-indices $I$ with $|I|\leq N+3$. 

Now fix a cutoff function $\chi \in C^{\infty}_c(\tilde \Omega)$ with $\chi \equiv 1 $ on $\Omega$ and define inductively $a^{(k)}, w^{(k)}, H^{(k)}$ by 
\begin{align} 
 H^{(k)} & = H +\chi E(a^{(k-1)})\label{d:hk}\\
 w^{(k)} &=\frac{1}{2}\mathcal C[\widehat{H^{(k)}}] \label{d:wk}\\
 a^{(k)} &= H_{22}-\partial_2w_2^{(k)}\,\label{d:ak}
\end{align}
which then yields \eqref{e:Kaellendecomp} due to $E(a^{(k-1)})_{22}= 0$. We now 

{\bf Claim:}  the following estimates hold:
\begin{align}
\| a^{(k)}-1\|_{0,\alpha}+ \|w^{(k)}\|_{1,\alpha}&\leq \frac{1}{2}, \label{e:decompestimate00}  \\
\|\partial_2^{l}\left (a^{(k)} - 1\right ) \|_{0,\alpha} + \|\partial_2^{l}w^{(k)}\|_{1,\alpha} &\leq C(k+1, \alpha)\sigma^{-l}\quad \text{ for }0\leq l \leq N+3-k, \label{e:decompestimate2}
\end{align}
and 
\begin{equation}\label{e:decompestimate3}
	\begin{split}
&\|\partial^{j+1}_1\partial_2^{l}\left (a^{(k)} - 1\right ) \|_{0,\alpha} + \|\partial_1^{j}\partial_2^{l+1}w^{(k)}\|_{1,\alpha}\\
\leq & C(k+1,\alpha)\sigma^{-(l+1)}\mu^{j} \quad \text{ for }1\leq l+j+1 \leq N+3-k\,.
\end{split}
\end{equation}
Observe that in \eqref{e:decompestimate3} we manage to trade a $\partial_1$ derivative on the coefficient $a^{(k)}$ (which is large due to the fast oscillation appearing in $E$) for a ``slower" derivative in $x_2$-direction. This is thanks to the expression \eqref{d:E1} for $a^{(k)}$, since now for example 
\[\partial_{1}a^{(k)} = \partial_1 H_{22}-\partial_{1}\partial_2 w^{(k)}_2 =\partial_1 H_{22}-  \frac{1}{2} \partial_1\mathcal C[\partial_2 \widehat{ H^{(k)}}]_2\,,\] so 
\[\|\partial_{1}a^{(k)}\|_{0,\alpha}\leq C\left ( \|H\|_{1,\alpha} + \|\mathcal C[\partial_2 \widehat{ H^{(k)}}]\|_{1,\alpha}\right )\leq C\left ( \|H\|_{1,\alpha} + \|\partial_2 H^{(k)}\|_{0,\alpha}\right )\,.\]
It is for this reason that we choose to subtract the $e_2\otimes e_2$ components in the definition of the error terms in \eqref{d:E1} and \eqref{d:E3}, since this leads to expression \eqref{d:ak}.

Now we show the claim by induction. Note that the  claim holds for $k=0$, since  $\sigma^{-1}\leq \mu $ (due to $\mu^{\alpha-1}\sigma^{-1}<\tau_*<1$). Assume that the claim is true for $k-1$ case. Then $E(a^{(k-1)})$ (and therefore $H^{(k)}$, $w^{(k)}$ and $a^{(k)}$) is well-defined thanks to \eqref{e:decompestimate00}. Let us show that \eqref{e:decompestimate00}-\eqref{e:decompestimate3} hold for $k$. Firstly, we find
 \begin{align*}
 	&\| a^{(k)}-1\|_{0,\alpha}+ \|w^{(k)}\|_{1,\alpha} \\
 	\leq &C\left (\|H-\id\|_{0,\alpha } + \|\chi\|_{0,\alpha}\|E(a^{(k-1)})\|_{0,\alpha}\right )\\
 	\leq &\frac{1}{4}+ C\|E(a^{(k-1)})\|_{0,\alpha} 
 \end{align*}
 because of \eqref{e:a00}. So we shall estimate $\|E(a^{(k-1)})\|_{0,\alpha}$. By the induction assumption  \eqref{e:decompestimate00}-\eqref{e:decompestimate3} with $k$ replaced by $k-1$, we immediately have  
 \[ \|\partial^{j_1}_1\partial^{l_1}_2 a^{(k-1)}\ldots \partial^{j_m}_1\partial_2^{l_{m}} a^{(k-1)}\|_{0,\alpha} \leq C \mu^{j_1-1}\sigma^{-(l_1+1)}\mu^{j_2}\sigma^{-l_2}\cdots \mu^{j_m}\sigma^{-l_m} \leq C\mu^{j}\sigma^{-(l+1)}\,\]   
 for any  $j_1>j_2\geq\ldots\geq j_m$ and $l_1\geq \ldots \geq l_m$ with $\sum_i j_i = j+1$, $\sum_i l_i = l$, using $k,m\leq N+3$.  Since $\|(a^{(k-1)})^{-m/2} \|_{0,\alpha}\leq C $ for $m\leq N+3$, we therefore find by Fa\`a di Bruno's formula 
 \begin{align}
  \|\partial_2^{l } \sqrt{a^{(k-1)}}\|_{0,\alpha} &\leq C(N,\alpha) \sigma^{-l} \quad  \quad \text{ for } l=0,\ldots, N+3-(k-1)\label{e:sqrtestimate0}\\ 
  \|\partial_1^{j+1}\partial_2^{l} \sqrt{a^{(k-1)}}\|_{0,\alpha} &\leq C(N,\alpha)\mu^{j}\sigma^{-(l+1)} \quad \text{ for } 0\leq l+j+1\leq N+3-(k-1).
  \label{e:sqrtestimate1}
  \end{align} 
Thus we can estimate
\begin{equation}\label{e:E1estimate}
	\begin{split} 
		&\|E_1(a^{(k-1)})\|_{0,\alpha} \leq C\left (\mu^{\alpha-1} \|M\|_{0}+ \mu^{-1}\|M\|_{0,\alpha}\right )\\
		\leq& C\left (\mu^{\alpha-1}\sigma^{-1} +\mu^{-1}\sigma^{-1-\alpha}\right ) \leq C\mu^{\alpha-1}\sigma^{-1}\,,
	\end{split}
\end{equation}
due to \eqref{e:Mestimates}, \eqref{e:sqrtestimate0} and $\sigma^{-1}\leq \mu$, and similarly 
\begin{equation}\label{e:E2estimate} \|E_2(a^{(k-1)})\|_{0,\alpha}\leq C \mu^{\alpha-1}\sigma^{-1} \end{equation} by the induction assumption. Also 
\begin{equation}\label{e:E3estimate} \|E_3(a^{(k-1)})\|_{0,\alpha}\leq C\mu^{\alpha-2}\sigma^{-2} \leq C\mu^{\alpha-1}\sigma^{-1} \,,\end{equation} 
by \eqref{e:sqrtestimate0} and \eqref{e:sqrtestimate1},  so that \eqref{e:decompestimate00} is true for sufficiently small $\tau_*(N,\alpha)$. As for the estimate \eqref{e:decompestimate2}, we have
\begin{align*} 
\|\partial_2^{l}\left (a^{(k)} - 1\right ) \|_{0,\alpha} + \|\partial_2^{l}w^{(k)}\|_{1,\alpha} &\leq C\left ( \|\partial_2^{l}(H-\id)\|_{0,\alpha} + \|\partial^{l}_2 E(a^{(k-1)})\|_{0,\alpha} \right )\\
& \leq C\left ( \sigma^{-l} + \sum_i \|\partial_2^{l}E_i(a^{(k-1)})\|_{0,\alpha}\right )\,.
\end{align*}
By the induction assumptions and \eqref{e:hestimates}, \eqref{e:Mestimates}, \eqref{e:sqrtestimate0}, \eqref{e:sqrtestimate1}, we get
\begin{align} \|\partial_2^{l}E_1(a^{(k-1)})\|_{0,\alpha}&\leq C(k,\alpha)\left (\mu^{\alpha-1}\|\partial^{l}_2(\sqrt{a^{(k-1)}}M)\|_0+\mu^{-1}\|\partial^{l}_2(\sqrt{a^{(k-1)}}M)\|_{0,\alpha}\right ) \nonumber\\
&\leq C(k,\alpha)\left ( \mu^{\alpha-1}\sigma^{-l-1} + \mu^{-1}\sigma^{-l-1-\alpha}\right )\leq C(k,\alpha)\mu^{\alpha-1}\sigma^{-1-l}\,,\label{e:derE1estimate}  \end{align} 
and 
\begin{equation}\label{e:derE2estimate}
\|\partial_2^{l}E_2(a^{(k-1)})\|_{0,\alpha} \leq C(k,\alpha)\mu^{\alpha-1}\sigma^{-(l+1)} \leq C(k,\alpha)\mu^{\alpha-1}\sigma^{-1} \sigma^{-l}\,,
\end{equation}
and a similar estimate for $E_3(a^{(k-1)})$. Putting these estimates together gives \eqref{e:decompestimate2}. The estimate \eqref{e:decompestimate3} follows analogously, using the product rule for the $\partial_1$ derivatives. Consequently, the claim is proved.

It now remains to show \eqref{e:differrors}. We also proceed by inductively showing that for any $k=1,\ldots, N$ it holds 
\begin{align}
\| a^{(k)}-a^{(k-1)}\|_{0,\alpha}+ \|w^{(k)}-w^{(k-1)}\|_{1,\alpha}&\leq C(N,\alpha)(\mu^{\alpha-1}\sigma^{-1})^k\,, \label{e:decompestimate20}
\end{align}
\begin{equation}\label{e:decompestimate22}
	\begin{split}
		&\|\partial_2^{l}\left (a^{(k)} - a^{(k-1)}\right ) \|_{0,\alpha} + \|\partial_2^{l} \left (w^{(k)}-w^{(k-1)}\right )\|_{1,\alpha} \\
		\leq &C(N, \alpha)(\mu^{\alpha-1}\sigma^{-1})^k\sigma^{-l} \quad \text{ for } 0\leq l \leq N+3-k
	\end{split}
\end{equation}
and 
\begin{equation}\label{e:decompestimate23}
	\begin{split}
		&\|\partial^{j+1}_1\partial_2^{l}\left (a^{(k)} - a^{(k-1)}\right ) \|_{0,\alpha} + \|\partial_1^{j}\partial_2^{l+1}\left (w^{(k)}-w^{(k-1)}\right )\|_{1,\alpha} \\
		\leq & C(N,\alpha)(\mu^{\alpha-1}\sigma^{-1})^k\sigma^{-(l+1)}\mu^{j}, \quad \text{ for }1\leq l+j+1 \leq N+3-k\,. 
		\end{split}
\end{equation}
Indeed, observe that 
\[a^{(1)}-a^{(0)}= \partial_{2}\left (w^{(1)}-w^{(0)}\right ) = \frac{1}{2}\partial_2\mathcal{C}[ \widehat {E(a^{(0)})}]\,\]
so that the estimates for $k=1$ follow from the corresponding ones above (see \eqref{e:E1estimate}--\eqref{e:E3estimate}, resp. \eqref{e:derE1estimate}--\eqref{e:derE2estimate}). Assume therefore that \eqref{e:decompestimate20}--\eqref{e:decompestimate23} hold for $k$. From 
\begin{align}a^{(k+1)}-a^{(k)}&= \partial_{2}\left (w^{(k+1)}-w^{(k)}\right ) = \frac{1}{2}\partial_2\mathcal{C}\left [ \chi\left (\widehat{ E(a^{(k)})}- \widehat {E( a^{(k-1)})}\right)\right ] 
\nonumber\\
& = \frac{1}{2}\mathcal{C}\left [\partial_2 \left (\chi\left (\widehat{ E(a^{(k)})}- \widehat{ E(a^{(k-1)})}\right )\right )\right ],\label{e:expressionfordifferences}\,\end{align}
it is apparent that in order to show \eqref{e:decompestimate20}--\eqref{e:decompestimate23} we need to estimate \[ \|\partial_2^{l}\left (E(a^{(k)})-E(a^{(k-1)})\right )\|_{0,\alpha} \text{   and } \|\partial_1^{j}\partial_2^{l+1}\left (E(a^{(k)})-E(a^{(k-1)})\right )\|_{0,\alpha} \]  in the given ranges of indices. To do this, recall that by definition \eqref{d:E} we have 
\[ E(a^{(k)})-E(a^{(k-1)}) = \sum_{i=1}^{3}\left (E_i(a^{(k)})-E_i(a^{(k-1)})\right )=\sum_{i=1}^3\frac{\bar\gamma_i(\mu x_1)}{\mu}A_i,\]
where 
\begin{align*}
	\bar\gamma_1=\gamma_1,\quad A_1=&(\sqrt{a^{(k)}}-\sqrt{a^{(k-1)}})(M-M_{22}e_2\otimes e_2),\\
	\bar\gamma_2=-\gamma_3,\quad 
	A_2=&\nabla({a^{(k)}}-{a^{(k-1)}})\odot e_1,\\
	\bar\gamma_3=-\frac{\gamma_1^2}{2\mu},\quad
	A_3=&\nabla \sqrt{a^{(k)}}\otimes\nabla\sqrt{a^{(k)}}-\nabla \sqrt{a^{(k-1)}}\otimes\nabla\sqrt{a^{(k-1)}}\\
	&+\left((\partial_2\sqrt{a^{(k-1)}})^2-(\sqrt{a^{(k)}})^2\right)e_2\otimes e_2.
\end{align*}
By the Leibniz rule we get for each term and indices $0\leq m+n\leq N$
\begin{equation}\label{e:Leibnitzestimate} \|\partial_1^{m}\partial_2^{n}\left(\frac{\bar\gamma_i(\mu x_1)}{\mu}A_i\right)\|_{0,\alpha} \leq C(N,\alpha)\sum_{j=0}^{m} \mu^{m-j-1}\left (\mu^{\alpha}\|\partial_1^{j}\partial_2^{n}A_i\|_{0}+\|\partial_1^{j}\partial_2^{n}A_i\|_{0,\alpha}\right ).\end{equation} 
Now for the second term, by induction assumption we immediately find 
\begin{align*}
	\|\partial_1^{m}\partial_2^{n}\left(\frac{\bar\gamma_2(\mu x_1)}{\mu}A_2\right)\|_{0,\alpha} \leq &C(N,\alpha)\sum_{l=1}^{m} \mu^{m-j-1}\mu^{\alpha}\sigma^{-(n+1)}\mu^{j}(\mu^{\alpha-1}\sigma^{-1})^k\\
	\leq& C(N,\alpha) (\mu^{\alpha-1}\sigma^{-1})^{k+1}\mu^{m}\sigma^{-n}\,.
	\end{align*}
To estimate the other terms, we write 
\[ \sqrt{a^{(k)}}-\sqrt{a^{(k-1)}} = \frac{a^{(k)}-a^{(k-1)}}{\sqrt{a^{(k)}}+\sqrt{a^{(k-1)}}}.\]
Thanks to \eqref{e:decompestimate00}, \eqref{e:sqrtestimate0} and \eqref{e:sqrtestimate1}, it holds 
\begin{align*} \|\partial_2^{l}\left (\sqrt{a^{(k)}}+\sqrt{a^{(k-1)}}\right )^{-1}\|_{0,\alpha}&\leq C(N,\alpha)\sigma^{-l}\\
\|\partial_1^{j+1}\partial_2^{l}\left (\sqrt{a^{(k)}}+\sqrt{a^{(k-1)}}\right )^{-1}\|_{0,\alpha}&\leq C(N,\alpha) \sigma^{-(l+1)}\mu^{j}\,.\end{align*}
Thus by induction assumption we find 
\begin{equation}\label{e:sqrtdifferences}
	\begin{split} \|\partial_2^{l}\left (\sqrt{a^{(k)}}-\sqrt{a^{(k-1)}}\right )\|_{0,\alpha}&\leq C(N,\alpha)(\mu^{\alpha-1}\sigma^{-1})^k\sigma^{-l}\\
\|\partial_1^{j+1}\partial_2^{l}\left( \sqrt{a^{(k)}}-\sqrt{a^{(k-1)}}\right )\|_{0,\alpha}&\leq C(N,\alpha) (\mu^{\alpha-1}\sigma^{-1})^k\sigma^{-(l+1)}\mu^{j}\,.
\end{split}
\end{equation}
Note that 
\begin{align*} 
	A_3 =&\nabla \sqrt{a^{(k)}}\otimes\left (\nabla \sqrt{a^{(k)}}-\nabla \sqrt{a^{(k-1)}}\right )+
\left (\nabla \sqrt{a^{(k)}}-\nabla \sqrt{a^{(k-1)}}\right )\otimes\nabla\sqrt{a^{(k-1)}}\\ 
& \quad-\left (\partial_2\sqrt{a^{(k)}}-\partial_2\sqrt{a^{(k-1)}}\right ) \left( \partial_2\sqrt{a^{(k)}}+\partial_2\sqrt{a^{(k-1)}}\right )e_2\otimes e_2\,.\end{align*}
Using \eqref{e:sqrtdifferences} and product rule, we can derive 
\[ \|\partial^{j}_1\partial_2^{n}A_3\|_{0,\alpha}\leq C(N,\alpha)(\mu^{\alpha-1}\sigma^{-1})^k\sigma^{-(n+1)-1}\mu^{j}\,.\]
Thus with \eqref{e:Leibnitzestimate} we find
\begin{align*} 
	\|\partial_1^{m}\partial_2^{n}\left(\frac{\bar\gamma_3(\mu x_1)}{\mu}A_3\right)\|_{0,\alpha} &\leq C(N,\alpha)\sum_{j=0}^{m}\mu^{m-j-2+\alpha}(\mu^{\alpha-1}\sigma^{-1})^k\sigma^{-(n+1)-1}\mu^{j} \\
	&\leq C(N,\alpha) (\mu^{\alpha-1}\sigma^{-1})^{k+1}\mu^{m-1}\sigma^{-(n+1)}\\
&\leq C(N,\alpha)(\mu^{\alpha-1}\sigma^{-1})^{k+1} \mu^{m}\sigma^{-n}\,.
\end{align*}
As for the first term, by assumption \eqref{e:Mestimates} and \eqref{e:sqrtdifferences} we can estimate for $\alpha\in [0,\beta[$
\[ \|\partial_1^{j}\partial_2^{n} A_1\|_{0,\alpha} \leq C(N,\alpha)(\mu^{\alpha-1}\sigma^{-1})^{k} \sigma^{-n-1-\alpha} \mu^{j}\,,\]
which yields
\begin{align*}
	\|\partial_1^{m}\partial_2^{n}\left(\frac{\bar\gamma_1(\mu x_1)}{\mu}A_1\right)\|_{0,\alpha}&\leq C(N,\alpha)(\mu^{\alpha-1}\sigma^{-1})^{k}\mu^{m-1}\sigma^{-n-1}\left (\mu^{\alpha}+\sigma^{-\alpha}\right )\\
	&\leq C(N,\alpha)(\mu^{\alpha-1}\sigma^{-1})^{k}
	\mu^{m-1+\alpha}\sigma^{-n-1} \\
&= C(N,\alpha)(\mu^{\alpha-1}\sigma^{-1})^{k+1}\mu^{m}\sigma^{-n}\,.\end{align*}
Combining the above estimates hence gives 
\begin{equation}\label{e:errordifferences}
	\|\partial_1^{m}\partial_2^{n}\left (E(a^{(k)})-E(a^{(k-1)})\right )\|_{0,\alpha}\leq C(N,\alpha)(\mu^{\alpha-1}\sigma^{-1})^{k+1}\mu^{m}\sigma^{-n}\,,\end{equation}
which together with the expression \eqref{e:expressionfordifferences} shows that \eqref{e:decompestimate20}--\eqref{e:decompestimate23} hold for every $k=0,\ldots, N$. Consequently \eqref{e:differrors} follows from \eqref{e:errordifferences}. 
The proof is then complete. 
\bibliographystyle{plain}

\end{document}